\documentclass[11pt,a4paper]{amsart}
\usepackage[T1]{fontenc}
\usepackage{lmodern}
\usepackage{amsmath,amssymb,amsthm,mathtools,booktabs,longtable}
\usepackage[a4paper,margin=30mm]{geometry}
\usepackage{microtype}
\usepackage{hyperref}
\hypersetup{colorlinks=true,linkcolor=blue,citecolor=blue,urlcolor=blue,pdftitle={The small Davenport constant of E2 times C3r for zero to three central factors},pdfauthor={Andreas Volkmann}}
\newtheorem{theorem}{Theorem}[section]
\newtheorem{lemma}[theorem]{Lemma}
\newtheorem{proposition}[theorem]{Proposition}

\theoremstyle{definition}

\newtheorem{remark}[theorem]{Remark}
\newcommand{\F}{\mathbb F_3}
\newcommand{\dd}{\mathsf d}

\newcommand{\one}{\mathbf 1}

\title[Small Davenport constants for three central factors]{The small Davenport constant of $E_2\times C_3^r$ for $0\le r\le3$}
\author{Andreas Volkmann}
\date{September 29, 2026}
\begin{document}
\begin{abstract}
Let $E_2$ be the extraspecial group of order $3^5$ and exponent three.
We prove that $\dd(E_2\times C_3^r)=2r+10$ for $0\le r\le3$.
The upper bounds follow from signed zero-block identities and two finite
statements in the four-dimensional symplectic space over $\mathbb F_3$.
The first supplies edge weights for all completable balanced triangles
on any indexed list of at most sixteen nonzero vectors. The second
supplies weights for the direction families that can occur in a critical
list with no central terms. We give complete coverage arguments, exact
certificate files, and separately implemented checking programs.
A compression argument removes any need for an induction through smaller
list lengths in the sixteen-term potential theorem. The formula for
$r\ge4$ remains open.
\end{abstract}
\maketitle

\section{Introduction}
An indexed sequence over a finite group is \emph{product-one-free} if no
nonempty selection of its occurrences can be ordered with product equal
to the identity. Its maximum possible length is the small Davenport
constant $\dd(G)$. The word ``indexed'' matters: repeated group values
remain separate available occurrences.

Let $E_2$ be the extraspecial group of order $3^5$ and exponent three.
We consider the conjectural formula
\begin{equation}\label{eq:conjecture}
 \dd(E_2\times C_3^r)=2r+10\qquad(r\ge0).
\end{equation}
The lower bound is elementary. The issue is to exclude a
product-one-free sequence of length $2r+11$.

\begin{theorem}\label{thm:main}
For $r=0,1,2,3$,
\[
 \dd(E_2\times C_3^r)=2r+10.
\]
In particular, the values for one, two, and three central direct factors
are $12$, $14$, and $16$, respectively.
\end{theorem}

The earlier paper \cite{VolkmannTwelve} excludes critical sequences
having exactly $2r-1$ central terms. That result leaves other central
occupancies untreated and does not assert the values in
Theorem~\ref{thm:main}. The present argument uses a different finite
potential theorem and a central-free compatibility theorem to cover all
occupancies for the stated values of $r$. We recall the elementary
ordering classification also used in \cite{VolkmannTwelve}, so that the
group-theoretic reduction below is self-contained. The previously
established Heisenberg family \cite{VolkmannH27} provides context; its
Davenport formula is not an input to this proof.

There are two computational inputs. The first is an edge-potential
statement for arbitrary indexed lists of at most sixteen nonzero
vectors in $\F^4$. The second is a compatibility classification for
families of directions of balanced triangles. Their proofs consist of
mathematical reductions to explicitly described finite enumerations,
followed by exact arithmetic over $\F$. The accompanying programs,
certificates, and execution records distinguish checking a stored weight
from establishing the coverage of the enumeration. No random sample is
used to establish either finite statement.

\section{Critical sequences and signed zero-block identities}\label{sec:framework}

\subsection{Coordinates and ordering corrections}
Put $a=r+4$ and write $V=\F^a=P\oplus R$, where $P=\F^4$ and
$\dim R=r$. Equip $V$ with an alternating form $\omega$ whose radical
is $R$ and whose restriction to $P$ is
\[
 \omega(x,y)=x_0y_1-x_1y_0+x_2y_3-x_3y_2.
\]
The group $G=E_2\times C_3^r$ can be written as $V\oplus\F$ with
\begin{equation}\label{eq:group}
 (v,t)(w,s)=(v+w,t+s-\omega(v,w)).
\end{equation}
Indeed $1/2=-1$ in $\F$; these are the usual class-two coordinates.
The derived group is $\{0\}\oplus\F$, and $(v,t)$ is central precisely
when the projection $\bar v\in P$ is zero.

For an indexed block $B$, let $\Omega(B)\subseteq\F$ be the set of
ordering corrections
\[
 -\sum_{j<k}\omega(v_{i_j},v_{i_k})
\]
as $(i_1,\ldots,i_{|B|})$ ranges over all orders of $B$.
Its possible products have vector coordinate $\sum_{i\in B}v_i$ and
last coordinate $\sum_{i\in B}t_i+\Omega(B)$.
Call $B$ \emph{full} if $\Omega(B)=\F$.
The nonorthogonality graph of $B$ joins two occurrences when their
pairing is nonzero. A triple $(i,j,k)$ is \emph{balanced} when
\begin{equation}\label{eq:balanced}
 \omega(v_i,v_j)=\omega(v_j,v_k)=\omega(v_k,v_i)\ne0.
\end{equation}
This property is independent of the chosen cyclic ordering.

\begin{lemma}[Nonfull blocks]\label{lem:nonfull}
A nonempty block is nonfull exactly in the following cases:
its graph is edgeless, consists of one edge and isolated vertices,
or consists of one balanced triangle and isolated vertices.
The correction set is $\{0\}$ in the first case and $\{1,-1\}$ in
the other two cases.
Consequently, a vector-zero block in a product-one-free sequence is
either commuting or a balanced triangle with isolated vertices.
\end{lemma}
\begin{proof}
A full subblock makes the entire block full: place the subblock
consecutively, so that its cross contributions do not depend on its
internal order. Two disjoint edges are full, since independently
swapping their adjacent endpoint pairs produces
$c+\{0,u\}+\{0,v\}=\F$ with $u,v\ne0$.
A graph with no two disjoint edges is a star or a triangle after its
isolated vertices are deleted. A star with at least two edges is full:
for its center and two leaves the two nonzero contributions can receive
their signs independently. For a triangle with cyclic pairings
$u,v,w\in\{1,-1\}$, the correction set is
\[
 \{\pm\tfrac12(u+v-w),\ \pm\tfrac12(u-v-w),\
   \pm\tfrac12(-u+v-w)\}.
\]
This is $\{1,-1\}$ precisely when $u=v=w$, and is $\F$ otherwise.
The edgeless and single-edge cases are immediate.
Finally, a full vector-zero block would have an ordering with product
one. A vector-zero block has weighted degree zero at every vertex,
because $\omega(v_i,\sum_Bv_j)=0$; hence its graph has no vertex of
degree one. This excludes the single-edge case.
\end{proof}

We say that a vector list is \emph{No-full} if every vector-zero
subblock is commuting or a balanced triangle with isolated vertices.
In such a triangle block its three nonisolated occurrences form its
unique \emph{core}. For a product-one-free group sequence, a nonempty
commuting vector-zero block has $\sum_Bt_i\ne0$, whereas a triangle
block has $\sum_Bt_i=0$.

\subsection{Boolean moments}
Let $S=(v_i)_{i=1}^n$ with $n=2a+3$. All signed counts in the rest of
the proof lie in $\F$. For an index set $Y$, define
\[
 N(Y)=\sum_{\substack{B\subseteq[n]\,;\ \sum_{i\in B}v_i=0\\Y\subseteq B}}
              (-1)^{|B|}.
\]
The empty block is included. For a balanced triple $K$ set
\begin{equation}\label{eq:mu}
 \mu_K=-N(K),\qquad \Theta(S)=\sum_{K\text{ balanced}}\mu_K.
\end{equation}
If $\mu_K\ne0$, at least one vector-zero block contains $K$.
Under No-full, that block has core $K$ and its other terms are
isolated and pairwise orthogonal. We call such a core \emph{supported}.

\begin{lemma}\label{lem:moments}
For $|Y|\le2$ one has $N(Y)=0$.
\end{lemma}
\begin{proof}
The polynomial
\[
 Q(X)=\prod_{h=1}^a\left(1-\left(\sum_{i=1}^n v_{i,h}X_i\right)^2\right)
\]
is the vector-zero indicator on $\{0,1\}^n$ and has degree at most $2a$.
For every polynomial $F$ of degree less than $n$,
\begin{equation}\label{eq:boolean}
 \sum_{B\subseteq[n]}(-1)^{|B|}F(\one_B)=0:
\end{equation}
each monomial omits a variable, and summing over that variable cancels
its contribution. Apply this to $Q(X)\prod_{i\in Y}X_i$.
\end{proof}

\begin{proposition}\label{prop:theta}
For every critical No-full vector list and every nonorthogonal pair
$e=\{i,j\}$,
\begin{equation}\label{eq:edge}
 \sum_{K\supseteq e\,;\ K\text{ balanced}}\mu_K=0.
\end{equation}
If the list comes from a product-one-free group sequence of length
$2a+3$, then
\begin{equation}\label{eq:theta-one}
 \Theta(S)=1.
\end{equation}
In particular, such a sequence has no occurrence with $v_i=0$.
\end{proposition}
\begin{proof}
Every vector-zero block containing $e$ has a unique triangle core
containing $e$. Partitioning these blocks by their third core occurrence
gives $N(e)=\sum_{K\supseteq e}N(K)$; use Lemma~\ref{lem:moments}.

Write $T(X)=\sum_i t_iX_i$. The polynomial $Q(X)(1-T(X)^2)$ has degree
at most $2a+2<n$. In \eqref{eq:boolean}, its empty-block contribution
is one. Its contribution at a nonempty commuting vector-zero block is
zero, and at a triangle block it is $(-1)^{|B|}$, by product-one
freeness and Lemma~\ref{lem:nonfull}. Therefore
$0=1+\sum_{B\text{ triangle block}}(-1)^{|B|}=1-\Theta(S)$.

If $v_j=0$, no balanced core contains $j$. Toggling $j$ pairs and
cancels the summands in $N(K)$ for every balanced $K$, so that
$\Theta(S)=0$, a contradiction.
\end{proof}

\subsection{The finite potential problem}
For an indexed list $U=(u_i)$ of nonzero vectors of $P$, let
$\mathcal A(U)$ consist of all dependent balanced triples and of
those independent balanced triples of sum $h$ for which $U$ contains
an occurrence of $-h$ or at least two occurrences of $h$.
For a balanced triple its sum is zero exactly in the dependent case.
Indeed, its Gram matrix has rank two and kernel generated by $(1,1,1)$;
if the three vectors are dependent their unique relation is their sum.
For an independent balanced triple of sum $h$, its common orthogonal
space in $P$ is the line $\langle h\rangle$.

An \emph{edge potential} for $U$ is a choice $\rho_{ij}\in\F$ on its
nonorthogonal index pairs such that
\begin{equation}\label{eq:potential}
 \rho_{ij}+\rho_{ik}+\rho_{jk}=1
 \quad\text{for every }\{i,j,k\}\in\mathcal A(U).
\end{equation}

\begin{lemma}\label{lem:admissible}
Let $U$ be the nonzero $P$-projections of all noncentral occurrences
of a critical No-full vector list. Every supported core belongs to
$\mathcal A(U)$. If $U$ has an edge potential, then $\Theta(S)=0$.
\end{lemma}
\begin{proof}
Only the independent case needs explanation. Let its sum in $P$ be
$h\ne0$, and take a vector-zero completion of the core. Every
noncentral isolate projects to $h$ or $-h$. Their sum is $-h$.
Thus at least one is $-h$, or at least two are $h$. None is a core
occurrence, since core vectors are not on their common orthogonal line.
Finally multiply \eqref{eq:edge} by $\rho_e$ and sum over edges.
Every supported core is counted with total weight one; unsupported
cores contribute zero. The result is $\Theta(S)=0$.
\end{proof}

\section{A potential theorem for sixteen indices}\label{sec:p16}

We use the notation $\mathcal A(U)$ and edge potentials from Section~\ref{sec:framework}.

\begin{theorem}\label{thm:p16}
Every indexed list $U$ of at most sixteen nonzero vectors in $P$
admits a potential.
\end{theorem}

We first specify the finite inputs used in the proof.  Their verification
is by exact arithmetic over $\mathbb F_3$; the accompanying programs
and certificates are part of the computational supplement.

\begin{lemma}[Finite direction certificates]\label{lem:p16-directions}
There is a weight on the $2160$ nonorthogonal unordered vector pairs
whose sum is zero on every dependent balanced triple and one on every
independent balanced triple.  There are $720$ triples of the former
kind and $5760$ of the latter.

Call a set $D$ of projective directions good if there are pair weights
whose sums are one on all dependent balanced triples and zero on all
independent balanced triples with sum direction in $D$.  Every set of
at most seven directions is good.  Every eight-element set is either
good or belongs to one of three specified symplectic orbits.  In each
of these three exceptional orbits, the set has at least two directions
orthogonal to no other member of the set.
\end{lemma}

The first assertion is certificate R.  For the remaining assertions,
the numbers of all symplectic orbits of $k$-element sets, for
$k=1,\ldots,8$, are respectively
\[
 1,\ 2,\ 5,\ 16,\ 43,\ 194,\ 785,\ 3140.
\]
All but three of the last orbits have stored good weights.  The
independent verifier constructs the projective symplectic group from
symplectic bases, checks every stored weight directly, checks that the
orbits are disjoint, and checks that their sizes sum to $\binom{40}{k}$
at every level.  Thus this is an exhaustive cover, rather than a
sampling of direction sets.  The three exceptional representatives,
using the vector code $x_0+3x_1+9x_2+27x_3$, are
\begin{align*}
 &\{10,13,14,15,17,28,30,31\},\\
 &\{10,12,15,16,27,28,30,36\},\\
 &\{5,10,13,14,28,30,31,36\}.
\end{align*}
Directions orthogonal to no other member include, respectively,
$\{10,15,31\}$, $\{10,16,36\}$, and $\{10,31\}$.
Only the stated cover and these isolation properties are used below;
no certificate of badness for an exceptional set is needed.

\begin{lemma}[Compression]\label{lem:p16-compression}
If an indexed list of at most sixteen vectors has no potential, then
there is such a list of exactly sixteen vectors with at most two
indices on each direction.
\end{lemma}
\begin{proof}
For every exact vector value $x$ keep one copy if $-x$ also occurs;
otherwise keep the smaller of two and the multiplicity of $x$.
The resulting sublist $U_0$ has at most two indices on every
direction.  It retains all exact vector values and every
admissibility condition for a balanced value triple: the presence of
$-h$ is retained, and, if $-h$ is absent, the condition that $h$
occurs at least twice is retained.

If $U_0$ admitted a potential, select one retained representative of
each exact value and pull the weights on representative pairs back
to all index pairs of $U$.  Every balanced triple has three different
values, and its supplementary indices lie on a line containing none
of its corners.  Hence its representative triple is still admissible,
and the pullback solves \eqref{eq:potential}.  Thus $U_0$ has no
potential.  This argument neither averages over copies nor divides
by their multiplicities.

If $|U_0|<16$, append vectors until the length is sixteen, always
using a direction with fewer than two existing indices.  Such
a line exists since there are forty directions.  Adding
indices only adds admissible triples, so it cannot restore
solvability.  The required bound of two indices per line is preserved.
\end{proof}

\begin{lemma}[A minimal inconsistent family]\label{lem:p16-circuit}
Let $C$ be an inclusion-minimal family of admissible value triples
for which the equations \eqref{eq:potential} are inconsistent.
Write $A_T$ for the incidence row of the three pairs of $T$.
There are nonzero coefficients $\lambda_T$ for every $T\in C$ such that
\begin{equation}
 \sum_{T\in C}\lambda_T A_T=0,
 \qquad \sum_{T\in C}\lambda_T\ne0.
 \label{eq:p16-dual}
\end{equation}
Moreover, the rows $A_T$ in every proper subfamily of $C$ are
linearly independent.
\end{lemma}
\begin{proof}
Inconsistency gives \eqref{eq:p16-dual} by linear algebra.  Minimality
forces each coefficient to be nonzero.  Suppose that a proper
subfamily admits a nonzero row relation $\beta$.  If its coefficient
sum is nonzero, it already witnesses inconsistency of a proper
subfamily.  If its sum is zero, subtract a suitable multiple of
$\beta$ from $\lambda$ to cancel one coefficient.  The resulting
relation still has nonzero coefficient sum and has smaller support,
again contradicting minimality.
\end{proof}

Passing from indexed triples to value triples causes no loss here.
One fixed representative of each value carries every admissible
value triple; conversely a solution on those representatives pulls
back to the indexed list.  Thus the two systems are solvable
simultaneously.

\begin{lemma}[A singleton direction]\label{lem:p16-singleton}
For a sixteen-index list satisfying the line bound in
Lemma~\ref{lem:p16-compression}, every minimal inconsistent family
$C$ has an independent member whose sum direction contains exactly
one index of the whole list.
\end{lemma}
\begin{proof}
Let $D$ be the sum directions of the independent triples in $C$.
Apply certificate R to \eqref{eq:p16-dual}.  The sum of the
independent coefficients is zero, so the dependent coefficients
have nonzero sum.  If $D$ were good, its good weights applied to the
same relation would make the latter sum zero.  In particular
$|D|\ge8$ by Lemma~\ref{lem:p16-directions}.  Every direction of
$D$ is occupied by an index, since the corresponding independent
triple is admissible.

If $|D|\ge9$, at least $2|D|-16\ge2$ of these occupied directions
are occupied only once.  If $|D|=8$, the set is one of the three
exceptional types.  Choose two directions $h,h'$ isolated for
orthogonality inside $D$, and corresponding triples in $C$.
Every corner of either triple lies outside $D$: it is orthogonal
to its own sum direction but is not on that direction.  The two
triples are different, since their sum directions differ, and
therefore use at least four distinct exact corner values, hence
at least four indices outside $D$.  At most twelve indices remain
on its eight occupied directions, so at least four of those
directions are occupied only once.  Either case gives the result.
\end{proof}

Choose such a singleton value $p$ and a triple
$K=\{k_1,k_2,k_3\}\in C$ on its direction.  Admissibility forces
$k_1+k_2+k_3=-p$.  By a symplectic similitude and a permutation of
the corners we may use the normal form
\begin{equation}
 k_1=(1,0,0,0),\quad k_2=(0,1,0,0),\quad
 k_3=(2,2,1,0),\quad p=(0,0,2,0).
 \label{eq:p16-normal}
\end{equation}
Indeed, the first two corners span a nondegenerate plane, and the
nonzero sum is orthogonal to that plane; these data extend to a
symplectic basis.  A similitude of multiplier $-1$ allows either
sign of the cyclic pairing.

The balance at each old pair in \eqref{eq:p16-dual} requires another
member of $C$ on that pair.  Denote its third corner by $l_i\ne k_i$
when the pair is $\{k_j,k_k\}$.  For a dependent replacement use an
empty list $T_i$.  For an independent replacement of sum $h_i$,
choose a supplementary $-h_i$ index, if available, or two $h_i$
indices.  Thus $|T_i|\in\{0,1,2\}$, a two-index $T_i$ has two
equal values, and
\begin{equation}
 l_i=k_i+p-\sigma(T_i).
 \label{eq:p16-pools}
\end{equation}

Here the three pools and the seven indices
$k_1,k_2,k_3,p,l_1,l_2,l_3$ are simultaneously available and
pairwise disjoint.  To check this essential point, put
$F_i=\langle k_j,k_k\rangle^\perp$.  Each nonempty pool lies on a
line in $F_i$ outside $\langle p\rangle$.  A pool on
$\langle p\rangle$ would have to consist of its sole index $p$,
forcing $l_i=k_i$.  Also $F_i\cap F_j=\langle p\rangle$ for
$i\ne j$, so different pools are disjoint.  A pool value is
orthogonal to two old corners, whereas every old or replacement
corner is nonorthogonal to two old corners.  The two pairs of old
corners intersect, ruling out equality of these values.  Finally,
replacement corners for two different old pairs demand opposite
pairing values at their common old corner, so they are distinct.
They are also distinct from $p$ and all old corners.  This proves
the asserted simultaneous index realization.

\begin{lemma}[The two positive star types]\label{lem:p16-stars}
Every sixteen-index list satisfying the line bound of two, in which
$p$ is the only index on its direction, and containing the seven fixed indices and the pools in
\eqref{eq:p16-normal}--\eqref{eq:p16-pools}, with sorted pool sizes
$(1,2,2)$ or $(2,2,2)$, admits a potential.
\end{lemma}

The capacity-two enumeration in \texttt{certify\_s16\_cap2.cpp}
checks these families directly, without invoking a smaller-list theorem.
It enumerates pool orbits under the stabilizer of
\eqref{eq:p16-normal}, followed by free multisets under each pool
stabilizer, subject to the capacity of two indices on each projective
line. Independent Burnside counts give $102{,}849{,}075$ templates
for pattern $122$ and $2{,}548{,}176$ for pattern $222$; both complete
runs attain exactly these counts.
Repeatedly deleting a triangle with a private pair reduces the
system to a labelled core. The $8980$ stored core certificates
satisfy $181{,}768$ core equations; reversing the deletions gives
full weights, checked in $265{,}122$ equations on their recorded
witness lists. The mathematical reconstruction and full coverage
are explained in Section~\ref{sec:reproducibility}.

For completeness, the finite circuit check needed after this star
reduction is specified next, including the argument that its search
cannot miss an inconsistent family.

Normalize the coefficient of $K$ in \eqref{eq:p16-dual} to one.
The three replacement coefficients each have two choices.  If all
pools are nonempty and at least two have size two, apply
Lemma~\ref{lem:p16-stars}.  The remaining sorted patterns are
\[
 000,001,002,011,012,022,111,112.
\]
Each $F_i\setminus\langle p\rangle$ has six nonzero vectors.
The normal-form stabilizer used for seed reduction has order eighteen;
it is generated by
\begin{align*}
 C(x)&=(-x_1,x_0-x_1+x_3,x_1+x_2,x_3),\\
 A(x)&=(x_0-x_1+x_3,-x_1,x_1+x_2-x_3,-x_3).
\end{align*}
It fixes $p$, permutes the old corners, and preserves the form up
to a nonzero scalar.  Direct generation of its closure and of the
pool orbits gives the following complete seed table:
\[
\begin{array}{c|rrrrrrrr|r}
\text{pattern}&000&001&002&011&012&022&111&112&\text{total}\\\hline
\text{raw pools}&1&6&6&36&36&36&216&216&553\\
\text{orbits}&1&1&1&9&12&9&12&36&81
\end{array}
\]
The eight choices of replacement coefficients yield $648$ starts.
The four initial incidence rows are independent: each replacement
triangle has a pair involving its own new corner, absent from the
other initial triangles, and the old triangle remains nonzero.

\begin{lemma}[Complete circuit check]\label{lem:p16-search}
None of these $648$ starts extends to a minimal inconsistent family
on a sixteen-index list with the specified singleton and line bounds.
\end{lemma}

Here is the complete search specification and its correctness
argument.  A state consists of selected value triples with assigned
nonzero coefficients, their incidence rows, their current pair
balances, and a multilist of at most sixteen available indices.
The latter always contains the singleton $p$, no $-p$, and at most
two indices on every other line.  Exact values receive stable
labels when first inserted.

While the selected incidence rows are independent, some current pair
balance is nonzero.  Choose such a pair.  Try every nonzero vector
as its third corner, subject to balance and to its triple not already
being selected.  Insert a missing corner.  If this independent
triple is not yet admissible, try both ways of making it admissible:
insert $-h$, or increase the multiplicity of $h$ to two.  Try both
nonzero coefficients for the new triple.

Reject only the following branches: an index or line bound is
violated; the selected rows acquire a consistent dependence; or the
entire current admissible system has already been found soluble.
An inconsistent dependence is a failure of the check.  At length
sixteen the whole system is solved.  At length fifteen all allowed
one-index extensions to length sixteen are solved, and the state
is rejected only if every extension is soluble.  Caching is used
only for identical multilists whose entire systems or entire sets
of one-index extensions have already been checked.

Suppose a family $C$ as in Lemma~\ref{lem:p16-circuit} existed.
The preceding normalization and star selection produce one of the
starts, unless Lemma~\ref{lem:p16-stars} already contradicts its
existence.  Follow the coefficients of \eqref{eq:p16-dual}.  A pair
with nonzero current balance must occur in some as yet unselected
triple of $C$; that triple, its actual coefficient, and suitable
supplementary indices from the ambient list are among the branches
tried.  No size or line bound discards this branch.  No proper
subfamily of $C$ has a row dependence, by
Lemma~\ref{lem:p16-circuit}.  If the available list reaches length
fifteen, its actual missing ambient index is among the one-index
extensions.  If it reaches length sixteen, it is the ambient list.
Neither complete system can be soluble, since it contains $C$.
Before reaching this boundary, each successful insertion raises
the row rank.  The process is therefore finite and must find either
the inconsistent dependence of $C$ or an inconsistent terminal
system.  This proves the completeness of the specification.

The stored computation has no such outcome.  One implementation
uses bitwise ternary elimination and additionally constructs explicit
terminal core weights.  Its accepted seed coverage consists of the
disjoint ranges $0$--$10$, $11$--$32$, and $33$--$80$.
An interrupted extra attempt at seed $11$ is not used in this
coverage.  The direct checker verifies all $21541$ saved core
certificates, $454565$ core equations, and $660527$ equations after
reconstructing full terminal weights, as well as the exact seed
coverage.  A second implementation uses ordinary byte matrices and
independently completes all $648$ starts, visiting $46308926$ states
and checking $57188554$ terminal lists, without an inconsistent
terminal system or a search limit.  It does not use any smaller-list
potential theorem.  Different traversal orders give different
terminal counts; no equality of those counts is required.

\begin{proof}[Proof of Theorem~\ref{thm:p16}]
An inconsistent list would, by Lemma~\ref{lem:p16-compression},
give an inconsistent sixteen-index list with the line bound.
Choose a minimal inconsistent value-triple family.  By
Lemma~\ref{lem:p16-singleton} it admits the normalized star
construction above.  The positive patterns are excluded by
Lemma~\ref{lem:p16-stars}, and every remaining pattern is excluded
by the complete check of Lemma~\ref{lem:p16-search}.
This contradiction proves the theorem.
\end{proof}

\begin{remark}
The orientation compression and subsequent padding in
Lemma~\ref{lem:p16-compression} are operations on the finite potential
problem only.  They are not operations on a product-one-free group
sequence and assert no preservation of original group heights.
Their role is to prove the potential theorem, which is subsequently
pulled back to the original indices without altering the group sequence.
\end{remark}

\section{The case without central projections}\label{sec:central-free}

We retain the critical length $n=2a+3$, No-full hypothesis, and signed counts of Section~\ref{sec:framework}.

For a balanced triple independent in $P$, its direction is the line
$h=\langle\sum_{i\in T}\bar v_i\rangle$. Its projected span has dimension
three, and its orthogonal complement in $P$ is exactly $h$.
In fact, writing two corners as $u,v$, the third is
$-u-v+\eta$ with $\eta\in\langle u,v\rangle^\perp\setminus\{0\}$;
then $\eta$ is the sum of the three corners.

\begin{lemma}\label{lem:cf-compatibility}
Assume in addition that $\bar v_i\ne0$ for all $i$.
Let $T,T'$ be balanced triples independent in $P$, with distinct
directions $h,h'$, each contained in a zero-sum block. Then either
$T$ and $T'$ have a common index, or $h\perp h'$ and a corner of one
triple projects onto the direction of the other.
\end{lemma}
\begin{proof}
Write the two zero-sum blocks as $T\cup J$ and $T'\cup J'$, where the
members of $J,J'$ are isolated vertices. The two blocks intersect:
otherwise their union would be a zero-sum block containing two
vertex-disjoint edges. Every member of $J$ projects to a nonzero vector
of $h$, and every member of $J'$ projects to a nonzero vector of $h'$.
Thus $J\cap J'=\varnothing$. A common index in $T\cap T'$ gives the first
alternative. A common index in $T\cap J'$ projects onto $h'$ and, as a
corner of $T$, is orthogonal to $h$, giving the second alternative.
The remaining case is symmetric.
\end{proof}

Here is the finite geometric statement needed below.  Its data and
complete verification programs accompany the paper.
Let $\mathcal E$ be the nonorthogonal unordered pairs of nonzero vectors
of $P$, and let $\mathcal T$ be the balanced triples of such vectors.
There are $2160$ pairs in $\mathcal E$ and $6480$ triples in $\mathcal T$,
of which $720$ are dependent and $5760$ are independent.
For weights $w:\mathcal E\longrightarrow\mathbb F_3$, write
$w(\Delta)=\sum_{e\subset\Delta}w(e)$.

Two independent projected triples with distinct directions $h,h'$ are
called compatible if they share a vector, or if $h\perp h'$ and one
contains a nonzero vector on the direction of the other.
Let $\mathcal D$ consist of the direction sets $D$ for which one can
choose, for every $h\in D$, a triple of direction $h$, such that all
chosen triples are pairwise compatible. This family is downward closed
and invariant under symplectic transformations.

\begin{proposition}[Finite geometric certificates]\label{prop:cf-finite}
\begin{enumerate}
\item There are weights $R:\mathcal E\to\mathbb F_3$ with
$R(\Delta)=0$ for every dependent balanced triple and $R(\Delta)=1$
for every independent balanced triple.
\item For every $D\in\mathcal D$, there are weights
$\rho_D:\mathcal E\to\mathbb F_3$ with $\rho_D(\Delta)=1$ for every
dependent balanced triple and $\rho_D(\Delta)=0$ for every independent
balanced triple whose direction belongs to $D$.
\end{enumerate}
\end{proposition}
\begin{proof}[Finite verification and its completeness]
All computations are over $\mathbb F_3$. Enumerate the $80$ nonzero
vectors, their $2160$ nonorthogonal pairs, and their $6480$ balanced
triples directly. The supplied $2160$ entries of $R$ satisfy the $6480$
claimed equations by direct substitution.

For the second statement, use the group on the $40$ directions generated
by the transvections $x\mapsto x+\omega(x,v)v$, one for each direction
$v$. Closing these permutations under composition gives $25920$
distinct permutations. A direction set is represented by the least
integer bit mask among its images under this explicitly generated
group. Starting with the empty set, form the canonical images of all
one-direction extensions of every accepted representative from the
preceding level. Decide membership in $\mathcal D$ by exhaustive
backtracking: there are $144$ triples per direction; whenever a triple
is selected, intersect the remaining candidate lists with its
compatibility lists. No branch with a possible compatible selection
is removed. Each returned selection is also checked directly.

This enumeration is complete by induction. If a set belongs to
$\mathcal D$, each of its one-direction deletions belongs to
$\mathcal D$; applying a group element to a deletion and its missing
direction realizes its orbit among the candidates at the next level.
The resulting counts are
\[
\begin{array}{c|rrrrrrr}
 k&1&2&3&4&5&6&7\\\hline
 \text{candidates}&1&2&5&16&43&194&784\\
 \text{accepted}&1&2&5&16&43&175&391
\end{array}
\]
\[
\begin{array}{c|rrrrrrr}
 k&8&9&10&11&12&13&14\\\hline
 \text{candidates}&2951&5140&3115&827&93&4&1\\
 \text{accepted}&376&175&47&7&2&1&0.
\end{array}
\]
For each of the $1241$ nonempty accepted representatives, the supplied
$2160$ edge weights satisfy every prescribed triangle equation by
direct substitution. For any group image, transport the weights by the
inverse transformation. The empty set is covered by restricting the
requirements of a singleton certificate. As level $14$ is empty and
$\mathcal D$ is downward closed, there are no later levels.

The computation is implemented in \texttt{e2univ.cpp}. A separate Python
implementation \texttt{verify\_dd.py} enumerates the vectors, triples,
group, canonical representatives and compatibility searches afresh;
it reproduces every level of the table. Certificate verification also
checks the complete list of representative masks, the edge list and
all field entries, and directly evaluates every prescribed equation.
\end{proof}

\begin{theorem}\label{thm:central-free}
If $S$ has the no-full property, has length $2a+3$, and satisfies
$\bar v_i\ne0$ for every index $i$, then $\Theta(S)=0$.
\end{theorem}
\begin{proof}
Let $D$ be the set of directions of supported triples independent in
$P$. For each direction select one such triple. A supported triple is
contained in a zero-sum block, because its signed number of extensions
is nonzero. Lemma~\ref{lem:cf-compatibility} implies that the selected
projected triples are pairwise compatible. Thus $D\in\mathcal D$.

For any weights $w$ on projected edges, Proposition~\ref{prop:theta} gives
\[
 \sum_{T\ \mathrm{balanced}}\mu_T w(\bar T)
 =\sum_{e\ \mathrm{noncommuting}}w(\bar e)
        \sum_{T\supset e}\mu_T=0.
\]
Choose $w=R+\rho_D$ from Proposition~\ref{prop:cf-finite}.
For every dependent triple $w(\bar T)=1$; for every supported independent
triple, $w(\bar T)=1+0=1$. Unsupported triples have coefficient zero.
The displayed sum is consequently $\Theta(S)$, proving the theorem.
\end{proof}

\paragraph{Coordinates and reproducibility.}
The theorem is basis independent. The geometric certificate programs
use coordinates $(x_0,x_1,x_2,x_3)$ with
$\omega(x,y)=x_0y_2-x_2y_0+x_1y_3-x_3y_1$ and integer encoding
$x_0+3x_1+9x_2+27x_3$. To pass from the coordinate convention
$(u_0,u_1,u_2,u_3)$ with form
$u_0v_1-u_1v_0+u_2v_3-u_3v_2$, use
$(x_0,x_1,x_2,x_3)=(u_0,u_2,u_1,u_3)$.
Direction representatives are the smaller of the codes of $x$ and
$-x$, ordered increasingly; their positions are the bits of a mask.
The file \texttt{dd\_representatives.txt} lists the $1241$ nonempty
representatives. The file \texttt{dcompat.weights} records its edge
order in an initial \texttt{EDGES} line and subsequently one weight
vector for each representative. The file \texttt{R\_edges.txt} supplies
$R$ as triples consisting of two vector codes and a field entry.
No solver needs to be trusted to check these weights: each required
triangle equation is evaluated directly.

\section{Completion of the Davenport bounds}\label{sec:completion}
The two finite statements imply a uniform restriction, even outside
the range of the main theorem.
\begin{proposition}\label{prop:occupancy}
If a product-one-free sequence of length $2r+11$ over
$E_2\times C_3^r$ exists, its number $q$ of central terms and number $L$
of noncentral terms satisfy
\[
 1\le q\le 2r-6,\qquad L\ge17.
\]
\end{proposition}
\begin{proof}
Proposition~\ref{prop:theta} gives $\Theta=1$. If $L\le16$, the
sixteen-term potential Theorem~\ref{thm:p16} and Lemma~\ref{lem:admissible} give
$\Theta=0$. Thus $L\ge17$. If $q=0$, Theorem~\ref{thm:central-free} gives
the same contradiction. Hence $q\ge1$, and
$q=2r+11-L\le2r-6$.
\end{proof}
\begin{proof}[Proof of Theorem~\ref{thm:main}]
For $r\le3$, Proposition~\ref{prop:occupancy} would give
$1\le q\le0$ for a product-one-free sequence of length $2r+11$.
Any longer product-one-free sequence would contain one of that length.
Consequently $\dd(E_2\times C_3^r)\le2r+10$.

For the reverse inequality, choose lifts $x_1,\ldots,x_4$ of a
symplectic basis of $E_2/E_2'$, a generator $z$ of $E_2'$, and
independent generators $c_1,\ldots,c_r$ of the direct central factors.
Take two copies of each of these $r+5$ elements.
In a product-one selection, projection to $G/G'$ forces the
multiplicity of every $x_i$ and $c_j$ to be zero modulo three.
Those multiplicities lie in $\{0,1,2\}$, so they are all zero.
The remaining multiplicity of $z$ must then also be zero. The
sequence is product-one-free and has length $2r+10$.
\end{proof}

\section{Reproducibility and scope}\label{sec:reproducibility}
\subsection{Exact finite certificates and reproducibility}
All arithmetic in the potential systems is performed in $\mathbb F_3$.
For a fixed indexed list, form the family of admissible balanced triangles
and attach one variable to every unordered index pair.  Each triangle
imposes the equation that the sum of its three edge variables is one.

\begin{lemma}[Private-edge reduction]
Suppose that an edge $e$ belongs to exactly one triangle $T$ in the
current equation system.  The system is consistent if and only if the
system obtained by deleting $T$ is consistent.  From any solution of the
smaller system one obtains a solution of the original system by setting
\[
  \rho_e=1-\sum_{f\subset T,\ f\ne e}\rho_f.
\]
\end{lemma}
\begin{proof}
The forward implication follows by restriction.  In the reverse
direction the variable $\rho_e$ occurs in no remaining equation, so the
displayed assignment changes none of those equations and satisfies the
deleted equation.  Successive deletions are reconstructed in reverse
order.
\end{proof}

The producer repeatedly removes a triangle having a private edge.
For every resulting nonempty labelled core it stores its complete list
of triangles, an explicit ternary edge-weight vector, and the first
indexed point list yielding that core.  Empty cores need no stored
weight vector.  The core cache uses an injective encoding: one bit for
each of the at most $\binom{16}{3}=560$ possible labelled triangles,
together with the number of core vertices in disjoint extra bits.
Thus equality of cache keys is equality of labelled systems, rather
than an unchecked graph-isomorphism assertion or a hash collision.

The independent Python checker enumerates admissible triangles anew,
checks every supplied core equation, identifies the core of each
stored witness list, reconstructs all original edge weights by the
lemma, and checks every original triangle equation.  It uses no linear
solver and imports no producer code.  A deliberately altered weight
is required to fail verification.  These checks verify all stored
weights.  Complete coverage of the searched configurations additionally
uses the enumeration algorithm and its mathematical completeness
argument; it is not inferred merely from the first witness of each
core.

\paragraph{Capacity-two stars.}
The two required sorted pool-size patterns are $122$ and $222$.
The new enumerator restricts each direction to at most two
indices, preserves the unique distinguished isolate, and contains no
appeal to any smaller potential theorem.  A separate Burnside
calculation, based on point cycles and remaining capacities on line
orbits, counts respectively $102{,}849{,}075$ and $2{,}548{,}176$
canonical templates.  Its generating factors for a line orbit of
length $\ell$ and remaining capacity $c\le2$ are
\[
 \sum_{j=0}^{c}(j+1)t^{\ell j}
 \quad\hbox{or}\quad
 \sum_{j=0}^{\lfloor c/2\rfloor}t^{2\ell j},
\]
according as the point action has two cycles of length $\ell$ or one
cycle of length $2\ell$.  Multiplication and extraction of the required
free-index degree give the fixed-point count; averaging over the pool
stabilizer gives the number of templates.  This integer Burnside
average is independent of the characteristic-three potential system.
Both complete producer runs finish with exit code zero.  The direct
checker verifies $8{,}980$ saved core certificates, $181{,}768$ core
equations and $265{,}122$ reconstructed triangle equations, with no
violations.  There are $8{,}509$ distinct labelled cores across the two
files.  The deliberately corrupted certificate is rejected.

\paragraph{Circuit search.}
The canonical starting lists comprise $81$ stars and $648$ nonzero
coefficient assignments.  The first exact search was completed in three
disjoint accepted parts: stars $0$--$10$, $11$--$32$, and $33$--$80$.
An interrupted additional branch in the first process is not used as a
completion certificate.  The checker verifies the exact starting-file
partition.  Across the three supplied core files it checks $21{,}541$
core certificates, $454{,}565$ core equations, and $660{,}527$ equations
on reconstructed indexed lists.  The second complete implementation
uses ordinary byte matrices and a different traversal order; its
recorded run completes all $648$ assignments in $46{,}308{,}926$ search
states with $57{,}188{,}554$ terminal systems.  Terminal counts from
different processes need not count globally distinct lists.

The terminal searches always enforce capacity two.  Consequently the
orientation-compressed length is exactly sixteen.  The supplied
producer header replaces the historical shortcut to a smaller
potential theorem by an explicit assertion of this fact; no smaller
potential theorem is required or called.  The independent byte-matrix
implementation has no such shortcut at all.

\paragraph{Reproduction.}
The ancillary README distinguishes the quick direct verification of
stored weights and starting-set coverage from complete search replay.
The quick Python checks require Python~3; only the projective-direction
orbit checker additionally requires NumPy.  Full replay uses a C++17
compiler supporting unsigned 128-bit integers for the first circuit
implementation, or the independent ordinary-matrix implementation.
All sources, inputs, stored weights, recorded outputs, and hashes needed
for the finite potential argument are included.

The restrictions in Proposition~\ref{prop:occupancy} do not exclude
all critical sequences when $r\ge4$. No value for those remaining
ranks is asserted here. In particular, a finite classification of
lists of length at most sixteen is not being extrapolated to
arbitrary length.

\subsection*{Use of computational and generative tools}
ChatGPT and Fable were used in developing proofs, checking algebraic
steps, writing and reviewing programs, and preparing this manuscript.
Separate implementations and checking programs were produced within
that AI-assisted workflow. They constitute distinct computational
checks, not independent human peer review. The mathematical claims
are the stated proofs and exact finite checks; unsuccessful searches
and informal model assessments are not evidence for them. The author
is responsible for the contents of the work.

\end{document}